\documentclass{article}
\usepackage[utf8]{inputenc}
\usepackage{amsmath}
\usepackage[utf8]{inputenc}
\usepackage[english]{babel}
\usepackage{mathtools}
\usepackage[arrow,matrix,curve,cmtip,ps]{xy}

\usepackage{amsthm}
\usepackage{comment}
\theoremstyle{definition}
\newtheorem{definition}{Definition}[section]
\usepackage[mathscr]{euscript}
\theoremstyle{plain}

\newtheorem{theorem}{Theorem}[section]

\newtheorem{corollary}[theorem]{Corollary}

\newtheorem{proposition}[theorem]{Proposition}

\newtheorem{lemma}[theorem]{Lemma}

\theoremstyle{definition}

\newtheorem{example}[theorem]{Example}

\usepackage[utf8]{inputenc}
\usepackage[english]{babel}

\usepackage{amsfonts}

\title{$\mathbb{N}$-Graph $C^{*}$-Algebras}
\author{Tim Schenkel}
\date{\today}

\begin{document}

\maketitle
\begin{abstract}
     In this paper we generalize the notion of a $k$-graph into (countable) infinite rank. We then define our $C^*$-algebra in a similar way as in $k$-graph $C^*$-algebras. With this construction we are able to find analogues to the Gauge Invariant Uniqueness and Cuntz-Krieger Uniqueness Theorems. We also show that the $\mathbb{N}$-graph $C^*$-algebras can be viewed as the inductive limit of $k$-graph $C^*$-algebras. This gives a nice way to describe the gauge-invariant ideal structure. Additionally, we describe the vertex-set for regular gauge-invariant ideals of our $N$-graph $C^*$-algebras. We then take our construction of the $\mathbb{N}$-graph into the algebraic setting and receive many similarities to the $C^*$-algebra construction.  
\end{abstract}
\section{Introduction}
\par Since Cuntz algebras were introduced they have become increasingly more general. First generalized into graph $C^*$-algebras and $k$-graph $C^*$-algebras. From there much work has been put into forming new algebras given similar structure. In \cite{BrownloweSimsVittadello} $p$-graphs were introduced. Many other generalizations have occurred including Leavitt Path Algebras, Kumjian-Pask Algebras, Cuntz-Pismer Algebras and more. In this paper we introduce another generalization, $\mathbb{N}$-graph $C^*$-algebras. 
\par We first generalize $k$-graphs with a factorization into $\mathbb{N}^{\mathbb{N}}_{0}$ and define a $C^*$-algebra defined using Cuntz-Krieger relations. We then show there is a directed system of $k$-graphs for which our $\mathbb{N}$-graph $C^*$-algebra is the inductive limit. Next we show that our $\mathbb{N}$-graph $C^*$-algebras can be viewed as an example of $P$-graph. We also give Gauge-Invariant Uniqueness and Cuntz-Krieger Uniqueness Theorems for our new algebra. Finally, we describe the gauge-invariant ideal structure of the $C^*$-algebra. 
In the next section we detour from the traditional route and follow in the steps of \cite{Schenkel} and give a vertex set description of the gauge-invariant regular ideals. We further show that aperiodicity is preserved when quotienting by a gauge-invariant regular ideal. 
\par Lastly, we consider an algebraic analogue of our $\mathbb{N}$-graph $C^*$-algebras. We largely follow the work of \cite{Kumjian Pask Algebra} to give Graded Uniqueness and Cuntz-Krieger Uniqueness Theorems. We also describe the basic, graded ideal structure. Lastly we look at the basic, graded regular ideals of these algebras.
It may be interesting to see a groupoid $C^*$-algebra construction or a Steinberg algebra construction of these algebras.

\subsection*{Acknowledgement}
I would like to sincerely thank my advisor, Dr. Adam Fuller, for his help with this research as part of my dissertation. Your continued guidance and support is greatly appreciated.
\section{Infinite rank graph $C^*$-algebras}
In this section we will introduce $\mathbb{N}$-graphs or (countably) infinite-rank graphs. They are closely modeled after $k$-graphs which were introduced in \cite{KumjianPask}. We will largely follow the work of \cite{KumjianPask} in our generalization of the $k$-graph to this setting.
We use the notation of $\mathbb{N}^{\mathbb{N}}_{0}$ to denote the subgroup of $\mathbb{N^N}$ such that there are only finitely many nonzero entries for each element, with addition defined entry-wise.
\begin{definition}
A $ \mathbb{N}$-graph $(\Lambda, d)$ consists of a countable category $\Lambda$ (with range and source maps $r$ and $s$ respectively) together with a functor $d: \Lambda \mapsto \mathbb{N}^{\mathbb{N}}_{0}$ satisfying the factorization property: For every $\lambda \in \Lambda$ with $d(\lambda) = m + n$, there are unique elements $\mu , \nu \in \Lambda$ such that $\lambda = \mu \nu $ and $d(\mu) = m, d(\nu) = n$. For $n \in \mathbb{N}^{\mathbb{N}}_{0}$ we write $\Lambda ^{n} = d^{-1}(n)$. A morphism between $\mathbb{N}$-graphs $(\Lambda _{1}, d_{1})$ and $(\Lambda_{2}, d_{2})$ is a functor $f: \Lambda _{1} \mapsto \Lambda _{2}$ compatible with degree maps. We will often refer to $\Lambda$ as a $\mathbb{N}$-graph with the degree map being omitted from notation.

\end{definition}

We will refer to elements of $\Lambda^0$ as vertices and elements of $mor(\Lambda)$ as paths. As with $k$-graphs, our factorization property gives us cancellation and we can identify $Obj(\Lambda )$ with $\Lambda^0.$

An $\mathbb{N}$-graph is row finite if for each $ n \in \mathbb{N}^{\mathbb{N}}_{0}$ and $v \in \Lambda^{0}$ the set $\Lambda^{n}(v) := \{ \lambda \in \Lambda^{n} : r(\lambda) = v \} $ is finite. Similarly $\Lambda $ has no sources if $\Lambda^{n}(v) \neq \emptyset$ for all $v \in \Lambda^{0}$ and $ n \in \mathbb{N}^{\mathbb{N}}_{0}$.
We will be assuming row-finite and no sources all paper unless otherwise stated.

\begin{definition}\label{CKR}
Let $\Lambda$ be a $\mathbb{N}$-graph, that is row-finite with no sources.
Then $C^{*}(\Lambda)$ is defined to be the universal $C^{*}$-algebra generated by a family  $$\{s_{\lambda}: \lambda \in \Lambda \}$$ of partial isometries satisfying:
\begin{enumerate}
    \item[(CK1)] $\{s_{v}: v \in \Lambda^{0} \}$ is a family of mutually orthogonal projections
    \item[(CK2)] $s_{\lambda \mu} = s_{\lambda} s_{\mu}$ for all $\lambda , \mu \in \Lambda$ such that $s(\lambda) = r(\mu)$, and 0 otherwise
    \item[(CK3)] $s_{\lambda}^{*} s_{\lambda} = s_{s(\lambda)}$ for all $\lambda \in \Lambda$. 
    \item[(CK4)] for all $v \in \Lambda^{0}$ and $n \in \mathbb{N}^{\mathbb{N}}_{0}$ we have $s_{v} = \sum _{\lambda \in v \Lambda^{n}} s_{\lambda}s_{\lambda}^{*}$ 
\end{enumerate}

\end{definition}
For $\lambda \in \Lambda$ define $p_\lambda = s_\lambda s_\lambda^*$. We note that for $v \in \Lambda^0$, $p_v = s_v$ as $s_v$ is a projection. We will often use $p_v$ instead of $s_v$ in our work to draw attention to when we are working with the vertex projection instead of the other isometries.
We also note that increasing finite sums of $p_v$'s forman approximate identity for $C^*(\Lambda).$
\begin{example}\label{ex}
\begin{enumerate}
    \item Define $\Omega = \Omega_{\mathbb{N}}$ as the small category with objects \\
    $Obj(\Omega) = \mathbb{N}^{\mathbb{N}}_{0}$, and morphisms $\Omega = \{(m,n) \in \mathbb{N}^{\mathbb{N}}_{0} \times \mathbb{N}^{\mathbb{N}}_{0} : m \leq n\},$ the range and source maps are given by $r((m.n)) = m,$ and $s((m,n))=n.$ Let $d: \Omega \rightarrow \mathbb{N}^{\mathbb{N}}_{0}$ be defined to be $d((m,n)) = n-m.$ It can then be shown that $C^*(\Omega) = \mathcal{K} (\ell^2(\mathbb{N}^{\mathbb{N}}_0)),$ where the rank one maps are given by $s_{(m,n)}$ and $s_{(m,n)}^*.$
    \item View the semigroup $\mathbb{N}^{\mathbb{N}}_{0}$ as a small category with $d: \mathbb{N}^{\mathbb{N}}_{0} \rightarrow \mathbb{N}^{\mathbb{N}}_{0}$ as the identity map. Then $(\mathbb{N}^{\mathbb{N}}_{0}, d)$ is a $\mathbb{N}$-graph. By the factorization property we get that $s_\lambda s_\mu = s_\mu s_\lambda$ for all $\mu, \lambda \in \mathbb{N}^{\mathbb{N}}_{0}.$ Further we get that $s_0 = Id$ so by KP3 and KP4, $ s_\lambda s_\lambda^* = Id = s_\lambda^*s_\lambda$. We will show that $C^*(\mathbb{N}^{\mathbb{N}}_{0}) \cong C^*(\mathbb{Z}^\infty)$, where $s_{e_i} $ for $i \in \mathbb{N}$ are the canonical unitary generators and $s_{e_i}^*$ make the generators for the inverses, and where $\mathbb{Z}^{\infty} $ is the subgroup of $\mathbb{Z^N}$ where each element contains only finitely many nonzero entries. Using a common argument with dual groups which can be found in \cite[Prop. VII.1.1]{Davidson}, we get $ C^*(\mathbb{N}^{\mathbb{N}}_{0}) \cong C^*(\mathbb{Z}^\infty) \cong C(\mathbb{T}^\omega),$ where $\mathbb{T}^{\omega} = \displaystyle\prod_{\mathbb{N}}\mathbb{T}$ is endowed with the product topology, as $\mathbb{Z}^\infty$ is the dual group of $\mathbb{T}^\omega$ \cite{Bell}.
\end{enumerate}
\end{example}

Using the same construction as that of \cite[Proposition 2.11]{KumjianPask} we show that there is a nontrivial $\ast $-representation of $(\Lambda, d)$.

\begin{theorem}\label{Rep}
Let $(\Lambda,d)$ be a $\mathbb{N}$-graph. Then there exists a representation $\{ S_{\lambda} : \lambda \in \Lambda \}$ of $\Lambda$ on a Hilbert space with all partial isometries $S_{\lambda}$ nonzero. 
\end{theorem}

\begin{proof}
Let $H = l^2(\Lambda^{\infty})$, then for $\lambda \in \Lambda$ define $S_{\lambda} \in B(H)$ by
$$ S_{\lambda}e_y =
\begin{dcases}
    e_{\lambda y} & s(\lambda) = y(0)\\
    0 & \text{otherwise}
\end{dcases}$$
where $e_y$ is the canonical basis for $H$. Notice that $S_\lambda$ is nonzero since \\$(s(\lambda))\Lambda^{\infty} \neq \{0\}$. It is now straightforward to check that it is a Cuntz-Kreiger $\Lambda$-family. 
\end{proof}

\begin{theorem}
Let $\Lambda$ be a $\mathbb{N}$-graph. Then for $\lambda , \mu \in \Lambda$ and $q \in \mathbb{N}^{\mathbb{N}}_{0}$ with $d(\lambda) , d(\mu) \leq q$ we have that $$s_{\lambda}^{*}s_{\mu} = \sum _{\lambda \alpha = \mu \beta \text{ and } d(\lambda \alpha) = q}s_{\alpha}s_{\beta}^{*} .$$
\end{theorem}

\begin{proof}
Same as \cite[Lemma 3.1]{KumjianPask}
\end{proof}
As there is a lot known about $k$-graph $C^*$-algebras, it will benefit us greatly to have a way to describe our $\mathbb{N}$-graph algebras in relation to $k$-graph $C^*$-algebras. It may be clear at this point that with the similarities between $k$-graph and $\mathbb{N}$-graphs, we can find $k$-graphs as subcategories of our $\mathbb{N}$-graphs.
\begin{definition}
Let $\Lambda$ be a $\mathbb{N}$-graph. Then we define  $$^{k}\Lambda := \{\lambda \in \Lambda : d|_{e^{n}}(\lambda) = 0 \text{ when } n>k \}, $$ where $d|_{e^{n}}(\lambda)$ is equal to the $n^{th}$ entry of $d(\lambda)$.
\end{definition}
\begin{theorem}
$^{k}\Lambda $ as defined above is a $k$-graph, with degree functor $d|_{\mathbb{N}^{k}} = \pi \circ d$, where $ \pi$ is the canonical projection of $\mathbb{N}^{\mathbb{N}}_{0} \rightarrow \mathbb{N}^{k}$. The range and source maps being $r|_{^k{\Lambda}}, s|_{^{k}\Lambda}$ respectively.

\end{theorem}
\begin{proof}
As $\Lambda$ is a category and ${}^k\Lambda $ is a subset, it is a subcategory. So we need only that the degree functor works. But as we have defined  $${}^k\Lambda = \{\lambda \in \Lambda : d|_{e^{n}}(\lambda) = 0 \text{ when } n>k \} . $$ So it is easy to check that the degree functor works and factorization is unique. Range and source maps are simply restricted to their inputs as all elements of degree $0$ are in ${}^k\Lambda$.
\end{proof}

Thus we can think of ${}^k\Lambda \subset \Lambda$ for each $k.$ It is also clear that for each $i \leq k$ that ${}^i\Lambda \subset {}^k\Lambda.$ Lastly by considering a morphism $\lambda \in \Lambda$ with $d(\lambda) \in \mathbb{N}^{\mathbb{N}}_{0}$ we get that $\lambda \in {}^k\Lambda$ for some $k$. Thus $\Lambda = \bigcup_{k\in \mathbb{N}} {}^k\Lambda.$
With this equivalence of categories in mind we turn to our ultimate goal of showing $C^*(\Lambda)$ as an inductive limit of $k$-graph $C^*$ algebras.
\begin{theorem}\label{c-k family}
Let $\Lambda$ be a row-finite $\mathbb{N}$-graph with no sources, and $(p,s)$ be a C-K $\Lambda$-family. Then there is a C-K ${}^k\Lambda$-family in $(p,s)$.
\end{theorem}
\begin{proof}
Consider the family $(p,s)_k := \{p_v , s_\lambda : v \in \Lambda^0, \lambda \in {}^k\Lambda \}.$ 
Then C-K (1-3) hold by ${}^k\Lambda$ being a subcategory. As for C-K(4) it holds from choice of $n \in \mathbb{N}^k$ giving all of the necessary elements contained in ${}^k\Lambda.$ So it is indeed a C-K ${}^k\Lambda$ family. 
\end{proof}

As ${}^k \Lambda$ is a $k$-graph, there exists a $k$-graph $C^{*}$-algebra, $C^{*}({}^k\Lambda)$ for each $k$.

\begin{corollary}
Let $\Lambda$ be a row-finite $\mathbb{N}$-graph with no sources, $i\leq k$ and $(p,s)$ be a C-K ${}^k\Lambda$-family. Then there is a C-K ${}^i\Lambda$-family in $(p,s)$.
\end{corollary}
Proof is the same as in Theorem \ref{c-k family}.
\begin{theorem}
The family of $C^{*}$-algebras $\{C^{*}({}^k\Lambda)\}$ form a directed system under the canonical inclusion maps. 
\end{theorem}
\begin{proof}
Define inclusion maps $\iota _{i,k} : C^{*}({}^k\Lambda) \rightarrow C^{*}({}^i\Lambda) $ by $ \iota_{i.k}(s_{\lambda}) = s_{\lambda}$. We need to show these maps are isomorphisms onto their range. First note that $\{\iota (s_{\lambda}) : \lambda \in {}^k\Lambda \}$ forms a C-K ${}^k\Lambda$ family (this is standard to check). Further each $\iota(p_v) \neq 0$ as each ${}^j\Lambda$ is row-finite with no sources as $\Lambda$ is. 

We now show there is a action of $\mathbb{T}^k$ on $\iota_{i.k}(C^{*}({}^i\Lambda)).$ For $t \in \mathbb{T}^k $ use the mapping $t \mapsto \alpha_{(t,1, 1, 1, . . . .,1)}$ where $\alpha_{(t,1,1,...,1)}$ is the automorphism from the gauge action on $C^{*}({}^i\Lambda)$.
Denote $\alpha_{(t,1,1,1,...,1)} = \beta_{t}$.

Now we claim that $\beta_{t} \circ \iota_{i.k} = \iota_{i,k} \circ \alpha_{t}.$ Notice that for an element \\$\sum _{i=1}^{n} s_{\lambda_{i}}s_{\mu_{i}}^{*} \in C^{*}({}^k\Lambda), $ that $ \beta_{t}\circ \iota_{i.k}(\sum^{n}_{i=1}s_{\lambda_{i}}s_{\mu_{i}}^{*}) \in \iota_{i.k} (C^{*}({}^k\Lambda).$ Since $\beta$ is continuous and $ C^{*}({}^k\Lambda)$ is closed we get that, $\beta_{t}(\iota_{i.k} (C^{*}({}^k\Lambda)) \subseteq \iota_{i.k} (C^{*}({}^k\Lambda)) $. Further note that $\beta_{\overline{t}}$ is an inverse. Now it is routine to check our claim. Further $\iota_{i.k}(p_{v}) = p_{v}  \neq 0. $ Thus by the gauge invariant uniqueness theorem \cite[Theorem 3.4]{KumjianPask}, $\iota_{i,k}$ is injective. 

It is easy to show that $\iota_{i.k} \circ \iota_{k,j} = \iota_{i,j}.$ Thus we have directed system.
\end{proof}
\begin{theorem} \label{limit}
Let $\Lambda$ be a row-finite $\mathbb{N}$-graph with no sources. Then $C^{*}(\Lambda) = \overline{\bigcup_{k\in \mathbb{N}} C^{*}({}^k\Lambda)}$.
\end{theorem}
\begin{proof}
We know there is a directed system $\{C^{*}({}^k\Lambda)\}$ with inductive limit \\$\overline{\bigcup _{k\in \mathbb{N}} C^{*}({}^k\Lambda)}$ by \cite[Theorem 6.2.4 i]{Rordam} since each inclucion map is injective. It remains to show that this limit is the universal $C^{*}$-algebra generated by a C-K $\Lambda$ family. Suppose that $\{t_{\lambda} : \lambda \in \Lambda \}$ is a C-K $\Lambda$ family in a $C^{*}$-algebra $B$. Then we wish to show there is a homomorphism $\pi : \overline{\bigcup _{k\in \mathbb{N}} C^{*}({}^k\Lambda)} \rightarrow C^{*}\{t_{\lambda} : \lambda \in \Lambda \}$. We know that $\Lambda = \bigcup_{k \in \mathbb{N}}  {{}^k\Lambda}$, so it is not hard to show that $C^{*}\{t_{\lambda} : \lambda \in \Lambda \} = \overline{\bigcup_{k \in \mathbb{N}} C^{*} \{t_{\lambda} \in {}^k\Lambda \} }$, using inclusion mappings to make a directed system. Also by Theorem \ref{c-k family},  $\{t_{\lambda}:\lambda \in {}^k\Lambda \}$ is a C-K ${}^k\Lambda$ family. Thus by universality of $C^{*}({}^k\Lambda)$ there exists $*$-homomorphisms $\pi_{k}: C^{*}({}^k\Lambda) \rightarrow C^{*}(\{t_{\lambda}:\lambda \in  {}^k\Lambda\}.$ We also note that $C^{*}\{t_{\lambda} : \lambda \in \Lambda \} = \overline{\bigcup_{k \in \mathbb{N}}{C^{*}\{t_{\lambda}:\lambda \in {}^k\Lambda \}}}$, under the inclusion mappings $\phi_{k} :C^{*}\{t_{\lambda}:\lambda \in {}^k\Lambda \} \rightarrow C^{*}\{t_{\lambda} : \lambda \in \Lambda\} $. Thus $\phi_{k} \circ \pi_{k}$ gives a $*$-homomorphism $C^{*}({}^k\Lambda) \rightarrow C^{*}\{t_{\lambda} : \lambda \in \Lambda\}$. Additionally, this mapping is commutative in the sense that $\phi_{k} \circ \pi_{k} = \phi_{k+1} \circ \pi_{k+1} \circ \iota_{k+1,k} $. Thus by Definition 6.2.2(ii) of \cite{Rordam}, there is one and only one morphism $\pi:\overline{\bigcup_{k\in \mathbb{N}} C^{*}({}^k\Lambda)} \rightarrow C^{*}\{t_{\lambda} : \lambda \in \Lambda \}$ so that $\pi \circ \iota_{k}  = \phi_k \circ \pi_{k}$. 
That $\pi$ is surjective follows from the fact that  $\pi_{k}$ is surjective and so we have $ C^{*}\{t_{\lambda} : \lambda \in \Lambda \} = \overline{\bigcup_{k \in \mathbb{N}} C^{*} \{t_{\lambda} \in {}^k\Lambda \} } = \overline{\bigcup_{k \in \mathbb{N}}\pi_{k}(C^{*}({}^k\Lambda)} = \pi(C^{*}(\Lambda))$. So by Proposition 6.2.4 of \cite{Rordam} $\pi$ is surjective. Thus $ \pi $ gives the mapping we want and $\overline{\bigcup _{k\in \mathbb{N}} C^{*}({}^k\Lambda)}$ is a universal $C^{*}$-algebra for a C-K $\Lambda$ family. So by uniqueness of the universal $C^{*}$-algebra $C^{*}(\Lambda) = \overline{\bigcup _{k\in \mathbb{N}} C^{*}({}^k\Lambda)}.$
\end{proof}

As with graph $C^*$-algebras and $k$-graph $C^*$-algebras there is a gauge action on our $\mathbb{N}$-graph algebra.
\begin{definition}
By the universal property of $C^{*}(\Lambda)$ there is a canonical action of $\mathbb{T}^{\omega}$ called the \emph{gauge action}: $\alpha : \mathbb{T}^{\omega} \mapsto \text{ Aut}C^{*}(\Lambda)$ defined for $t = (t_{1}, t_2 , . . . ..) \in \mathbb{T}^{\omega}$ and $s_{\lambda} \in C^{*}(\lambda) $ by $$ \alpha_{t}(s_{\lambda}) = t^{d(\lambda)}s_{\lambda}.$$ Where $t^{m} = t_{i}^{m_{1}}t_{2}^{m_{2}}....$ for $m = (m_1, m_2, ... ) \in \mathbb{N^{N}}_{0}.$ 
\end{definition}
\begin{theorem}
The gauge action on $C^{*}(\Lambda)$ is strongly continuous in the product topology.

\end{theorem}
\begin{proof}

Take a sequence $\{t_{n}\} \in \mathbb{T}^{\omega} $converging to $t.$ Then we need to show that $\alpha_{t_{n}}$ converges to $\alpha_{t}.$ Take $a \in C^{*}(\Lambda). $ Then we know that $a \in \overline{span}\{s_{\lambda}s_{\mu}^{*} : s(\lambda) = s(\mu)\}.$ Take $\epsilon > 0.$ Then   $$||\alpha_{t_{n}}(a) - \alpha_{t}(a)|| \leq||\alpha_{t_{n}}(a) - \alpha_{t_{n}}(\sum s_{\lambda_i}s_{\mu_i{}}^{*})|| + $$ $$|| \alpha_{t_{n}}(\sum s_{\lambda_i}s_{\mu_i{}}^{*}) - \alpha_{t}(\sum s_{\lambda_i}s_{\mu_i{}}^{*})||+ ||\alpha_{t}(\sum s_{\lambda_i}s_{\mu_i{}}^{*}) - \alpha_{t}(a)||$$
By strong continuity of the gauge action on $k$-graph $C^{*}$-algebras and the fact that $\alpha_{t}$ and $\alpha_{t_{n}}$ are automorphisms, and density of $\{s_{\lambda}s_{\mu}^{*} : s(\lambda) = s(\mu)\}$ in $C^{*}(\Lambda)$ we can find an $n \in \mathbb{N}$ and $\sum s_{\lambda_i}s_{\mu_i{}}^{*} $ so that each of these are less than $\epsilon / 3.$

\end{proof}
\begin{corollary}
 There is a faithful conditional expectation $\Phi : C^*(\Lambda) \mapsto C^*(\Lambda)^\gamma$ defined by $\Phi(x) = \int_{\mathbb{T}^k} \alpha_t (x) \,dt$, where $C^*(\Lambda)^\gamma $ is the fixed point algebra of the gauge action.
\end{corollary}
\begin{proof}
Apply \cite[Corollary 5]{Raeburn} to our strongly continuous gauge action.
\end{proof}

We show in the next theorem that we can consider these $\mathbb{N}$-graph $C^*$-algebras as $P$-graph $C^*$-algebras in the sense of \cite{BrownloweSimsVittadello}.
\begin{theorem}
 Let $\Lambda$ be a row-finite $\mathbb{N}$-graph with no sources. Then the co-universal $C^*$-algebra $C_{min}^* (\Lambda)$ obtained from \cite[Theorem 5.3]{BrownloweSimsVittadello} by regarding $\Lambda$ as an $\mathbb{N}^{\mathbb{N}}_{0}$-graph is canonically isomorphic to the $\mathbb{N}$-graph $C^*$-algebra $C^*(\Lambda).$
\end{theorem}

\begin{proof}
By definition, $C^*(\Lambda) $ is generated by a representation $t$ of $\Lambda$ in the sense of \cite[Definition 5.1]{BrownloweSimsVittadello} (though the definitions are slightly different it can be seen in \cite[Appendix B]{RaeburnSims} that they are equivalent).
We know that $t_v \neq 0$ for all $v \in \Lambda^0$ as we have a nonzero representation, and that there is an action $\alpha$ of $\mathbb{T}^\omega$ on $C^*(\Lambda)$ satisfying $\alpha_z (t_\lambda) = z^{d(\lambda)}t_\lambda$ for all $\lambda \in \Lambda.$ Under action-coaction duality, $\alpha$ determines a coaction $\gamma$ of $\mathbb{Z}^{\mathbb{N}}_0$ satisfying $\gamma(t_\lambda )= \ t_\lambda \otimes U_{d(\lambda)}$ for all $\lambda \in \Lambda.$
This coaction is normal because $\mathbb{Z}^{\mathbb{N}}_0$ is abelian and hence amenable. Additionally, the $t_\lambda$ satisfy that \[ \prod _{\lambda \in E} (t_v - t_\lambda t_\lambda^*) = 0 \text{ for all } v \in \Lambda^0 \text{ and finite exhaustive } E \subset v\Lambda\] as $E$ is finite and hence belongs to a subset $v {}^k\Lambda$ for some $k$. Thus by \cite[Theorem 5.3]{BrownloweSimsVittadello} the result follows.
\end{proof}
\section{Uniqueness theorems for $C^*(\Lambda)$}
In this section we give gauge-invariant and Cuntz-Krieger uniqueness theorems for our $\mathbb{N}$-graph $C^*$-algebras. We rely on a combination of using our inductive limit construction and by following in the steps of \cite{HazlewoodRaeburnSimsWebster} where appropriate.

\begin{theorem}\label{GIUT}
Let $B$ be a $C^{*}$-algebra and $\pi: C^{*}(\Lambda) \rightarrow B$ be a homomorphism and let $\beta: \mathbb{T}^{\omega} \rightarrow \text{Aut}(B)$ be an action such that $\pi \circ \alpha_t = \beta_t \circ \pi$ for every $t \in \mathbb{T}^{\omega}.$ Then $\pi$ is faithful if and only if $\pi(p_{v}) \neq 0$ for every $v \in \Lambda^0.$
\end{theorem}
\begin{proof}
If $\pi(p_{v}) = 0 $ for some $v \in \Lambda^0$ then clearly $\pi$ is not faithful. 
To see the other direction we note that if $\pi$ and $\beta$ are as described above then $\beta|_{\mathbb{T}^{k}}$ is an action of $\mathbb{T}^{k} \mapsto \text{Aut}(B)$ such that $\pi \circ \alpha_{t} = \beta_{t} \circ \pi$ for every $t \in \mathbb{T}^{k}. $ Thus by the gauge invariant uniqueness theorem, \cite[Theorem 3.4]{KumjianPask}, $\pi|_{C^{*}({}^k\Lambda)}$ is faithful for each $k$. Thus we have injective maps $\pi|_{C^{*}({}^k\Lambda)}: C^{*}({}^k\Lambda) \rightarrow B  $ and $\iota_{k}: C^*({}^k\Lambda) \rightarrow C^{*}(\Lambda).$ Thus by \cite[Proposition 6.2.4]{Rordam}, $\pi$ is injective.
\end{proof}
\begin{example}
To finish Example \ref{ex}(ii) we use the gauge invariant uniqueness theorem. Note that with the unitary generators as described we get that there is a C-K $\mathbb{N}^{\mathbb{N}}_0$ family in $C^*(\mathbb{Z}^{\infty})$. Further there is a gauge action which commutes with our homomorphism given by the univeral property. Hence our homomorphism is injective by Theorem \ref{GIUT}. That it is surjective comes from mapping onto all of the generators. Thus we have $C^*( \mathbb{N}^{\mathbb{N}}_0) \cong C^*(\mathbb{Z}^{\infty}) $ as described.
\end{example}
\begin{definition}
We say $\Lambda$ is \emph{aperiodic} if for each vertex $v \in \Lambda^{0}$ and each pair $n \neq m \in \mathbb{N^{N}}_{0}$ there is a path $\lambda \in v \Lambda$ such that $d(\lambda) \geq m \vee n$ and $$\lambda(m, m+d(\lambda)-(m \vee n)) \neq \lambda(n, n+ d(\lambda)-(m \vee n))$$ 
\end{definition}
We dedicate the remainder of this section to proving a Cuntz-Krieger uniqueness theorem for $C^*(\Lambda).$ We will follow largely the steps of \cite{HazlewoodRaeburnSimsWebster}.
\begin{lemma}\label{111}
Let $(\Lambda, d) $ be an aperiodic $\mathbb{N}$-graph with no sources. Suppose that $v \in \Lambda^0$ and $l \in \mathbb{N}^{\mathbb{N}}_{0}$. Then there exists $\lambda \in \Lambda$ such that $r(\lambda)=v, d(\lambda) \geq l$ and $ \alpha, \beta \in \Lambda v , d(\alpha), d(\beta) \leq l $ and $\alpha \neq \beta $ implies $(\alpha \lambda)(0,d(\lambda)) \neq (\beta \lambda)(0,d(\lambda)).$
\end{lemma}
Proof is the same as in \cite[Lemma 6.2]{HazlewoodRaeburnSimsWebster}.
\begin{lemma}
Suppose that $\Lambda$ is a row-finite aperiodic $\mathbb{N}$-graph with no sources. Let $\{t_{\lambda} : \lambda \in \Lambda \} $ be a Cuntz-Krieger $\Lambda$-family in a $C^{*}$-algebra $B$ such that $t_v \neq 0$ for all $v \in \Lambda^0.$ Let $F$ be a finite subset of $\Lambda$ and let $a: (\mu, \nu) \mapsto a_{\mu , \nu }$ be a $\mathbb{C}$-valued function on $ F \times F$ such that $s(\mu) = s(\nu)$ whenever $a_{\mu , \nu} \neq 0.$ Then
$$ || \sum _{\mu, \nu \in F} a_{\mu, \nu}t_\mu t_\nu^* || \geq || \sum_{\mu,\nu \in F, d(\mu) = d(\nu)}a_{\mu, \nu}t_\mu t_\nu^* || $$
\end{lemma}
Follows the same as \cite[Proposition 6.4]{HazlewoodRaeburnSimsWebster}.
\begin{theorem}
Let $\Lambda$ be a row-finite, aperiodic $\mathbb{N}$-graph with no sources. Suppose that $\{t_\lambda : \lambda \in \Lambda \}$ is a Cuntz-Krieger $\Lambda$-family, and let $\pi$ be the homomorphism of $C^{*}(\Lambda)$ such that $\pi (s_\lambda) = t_\lambda$ for all $\lambda \in \Lambda$. If each $t_v \neq 0$, then $\pi$ is faithful.
\end{theorem}
\begin{proof}
We first want to show that $\pi$ is injective on $C^{*}(\Lambda)^{\gamma}$, the fixed point algebra of the gauge action. Since $C^*(\Lambda)=\overline{span}\{s_\lambda s_\nu ^* : s(\lambda) = s(\nu) \}$ the fixed point algebra will be the span closure of the generators which are fixed by the gauge action. Thus $C^{*}(\Lambda)^{\gamma} = \overline{span}\{s_\mu s_\nu^* : d(\mu)=d(\nu)\}.$ Set $A = \overline{\bigcup_{k=1}^{\infty} C^{*}({}^k\Lambda)^{\gamma}}.$ Then $A$ is the inductive limit of the AF-algebras $ \{C^{*}({}^k\Lambda)^{\gamma}\}$. Since as in proof of \cite[Theorem 6.1]{HazlewoodRaeburnSimsWebster}, $C^{*}({}^k\Lambda)^{\gamma} = \overline{span}\{s_\mu s_\nu^* : d(\mu) = d(\nu), \mu , \nu \in {}^k\Lambda \} $, we get that $A \cong C^{*}(\Lambda)^{\gamma}$, as both contain the span and are closed. Now let $\pi_{k}$ denote the restriction of $\pi$ to $C^{*}({}^k\Lambda)^{\gamma}$ for each $k$. Then as in the first paragraph of \cite[Theorem 3.4]{KumjianPask}, each $\pi_k$ is injective. Let $\iota_k^\gamma$ denote the restriction of $\iota_k$ to $C^{*}({}^k\Lambda)^{\gamma}.$ Then $\iota_{k}^{\gamma}$ is injective. Thus by \cite[Proposition  6.2.4]{Rordam}, $\pi$ is injective on $C^{*}(\Lambda)^{\gamma}.$ 
The rest of the proof follows as in \cite[Theorem 6.1]{HazlewoodRaeburnSimsWebster}. 
\end{proof}
\section{Ideal structure of $C^*(\Lambda)$}
In this section we give a vertex set description for the gauge-invariant ideals of $C^{*}(\Lambda)$. Our work relies largely on the inductive limits of $k$-graph $C^*$-algebras. We again find that these close similarities between the two types of categories yields similar results in the structure of the resulting $C^*$-algebras.
\begin{definition}
Let $(\Lambda, d)$ be a row-finite $\mathbb{N}$-graph with no sources.
We say a subset $H$ of $\Lambda^{0}$ is \emph{hereditary} if $\lambda \in \Lambda$ and $r(\lambda) \in H$ imply $s(\lambda) \in H$. We say that $H$ is \emph{saturated} if for $v \in \Lambda^{0}, \  s(v\Lambda^{n}) \subset H$ for some $n \in \mathbb{N}^{\mathbb{N}}_{0}$ implies $v \in H$.
\end{definition}
\begin{definition}
For a subset $H$ of $\Lambda^{0}$ define $I(H)$ be the closed ideal in $C^{*}(\Lambda)$ generated by $ \{ p_{v} : v \in H \}.$
For $I$ an ideal in $C^{*}(\Lambda)$ define
$$H(I) : = \{ v\in \Lambda ^{0} :p_{v} \in I \}$$
\end{definition}
\begin{theorem}\label{lattice}
Let $(\Lambda,d)$ be a row-finite $\mathbb{N}$-graph with no sources. Then the map $H \mapsto I(H)$ is an isomorphism of the lattice of saturated hereditary subsets of $\Lambda^0$ onto the lattice of closed gauge-invariant ideals of $C^{*}(\Lambda)$. 
\end{theorem}
\begin{proof}
Let $I$ be a gauge-invariant ideal of $C^{*}(\Lambda).$ We first want to show that $I$ is generated by a saturated and hereditary set of vertex projections. We first note that since $C^{*}(\Lambda) = \overline{\bigcup_{k \in \mathbb{N}} C^{*}({}^k\Lambda)}$ we get $I = \overline{\bigcup_{k \in \mathbb{N}}I \cap C^{*}({}^k\Lambda)}$ by \cite[Exercise 7.3]{Rordam}. As an ideal intersected with a subalgebra is an ideal of the subalgebra, we need only show $ I \cap C^{*}({}^k\Lambda)$ is gauge-invariant. Take $t \in \mathbb{T}^{\omega}$. Then as $I$ is gauge-invariant we get that $\alpha_t(I) \subseteq I$. Now note that $\alpha_t(C^{*}({}^k\Lambda)) = C^{*}({}^k\Lambda).$ Thus $I \cap C^{*}({}^k\Lambda)$ is a gauge-invariant ideal of $C^{*}({}^k\Lambda).$ Thus by \cite[Theorem 5.2]{RaeburnSimsYeend}, $I \cap C^{*}({}^k\Lambda) = I(H_k)$ for some hereditary and saturated set $H_k$ of ${}^k\Lambda^0.$  As this is true for all $k$ we get that $I = \overline{\bigcup_{k \in \mathbb{N}}I(H_k)}.$ Since $d(v) = 0 $ for every $v \in \Lambda^0$, we get that $H_k = H_i$ for every $i,k \in \mathbb{N}.$ Thus there is a saturated and hereditary set $H$ of $\Lambda$ so that $ I(H_k) = \overline{span}\{s_\alpha s_{\beta}^{*} \in C^{*}({}^k\Lambda): s(\alpha) = s(\beta) \in H\}.$ Thus $I =\overline{span}\{s_\alpha s_{\beta}^{*} \in C^{*}(\Lambda): s(\alpha) = s(\beta) \in H\} $.

We next need to show that for $H$ a saturated and hereditary set of vertices in $\Lambda$ that $I(H)$ is gauge-invariant. Note that $H \cap {}^k\Lambda = H$ for every $k \in 
\mathbb{N}.$ We claim $H$ is saturated and hereditary in ${}^k\Lambda. $ Indeed, that it is hereditary is clear. To see that it is saturated we note that  ${}^k\Lambda v \subseteq \Lambda v.$ Denote $I(H)_k$ as the ideal of $C^{*}({}^k\Lambda)$ generated by $\{p_v \in C^{*}({}^k\Lambda) : v \in H\}. $ By \cite[Theorem 5.2]{RaeburnSimsYeend} each $I(H)_k$ is gauge invariant in $C^{*}({}^k\Lambda).$ Define $I_k := I(H) \cap C^{*}({}^k\Lambda).$ By definition $I(H)_k \subseteq I_k$ since $\{p_v : v \in H \} \subset I_k$ for all $k$. We want to show that $I(H)_k = I_k. $ See that $I(H)_k = \overline{span} \{s_\alpha s_\beta^{*} \in C^{*}({}^k\Lambda): s(\alpha) = s(\beta) \in H \}$,  and $ I(H) = \overline{span} \{s_\alpha s_\beta^* \in C^*(\Lambda) : s(\alpha) = s(\beta) \in H \}$ as $C^*(\Lambda) = \overline{span}\{s_\alpha s_\beta^* : s(\alpha) = s(\beta)\}$ and $I(H)$ is generated by $\{p_v : v\in H\}.$ Thus $I_k = I(H)_k.$ So we have that $I(H) = \overline{ \bigcup_{k \in \mathbb{N}} I(H)_k}.$ Since each $I(H)_k$ is gauge-invariant, so is $I(H)$.
\end{proof}
\begin{theorem}\label{quotient}
Let $\Lambda$ be a row-finite $\mathbb{N}$-graph with no sources. Suppose that $H$ is saturated and hereditary. Then $\Lambda \setminus H$, the small category with objects $ \Lambda^{0} \setminus H$, and morphisms $ \{ \lambda \in \Lambda : r(\lambda) \text{ and } s(\lambda) \in \Lambda^{0} \setminus H \}$, with the factorization property $d$ inherited from $(\Lambda, d)$,
    is a row-finite $\mathbb{N}$-graph with no sources, 
    and $C^{*}(\Lambda)/I(H)$ is canonically isomorphic to $C^{*} (\Lambda \setminus H)$.
\end{theorem}

\begin{proof}
That $\Lambda \setminus H$ is a category is clear by definition. Source maps are clear based on definition and range maps follow by $H$ being hereditary. It remains to show that the factorization property remains. Take $\lambda \in \Lambda \setminus H$ and suppose $d(\lambda) = p + q.$ Then we know there exists $\mu, \nu \in \Lambda$ so that $\lambda = \mu \nu, d(\mu) = p$ and $d(\nu) = q.$ We know that $s(\nu) = s(\lambda) \notin H$ by choice of $\lambda$. As $H $ is hereditary, $r(\nu) = s(\mu) \notin H.$ Thus $\mu, \nu \in \Lambda \setminus H.$ So $\Lambda \setminus H $ is a $\mathbb{N}$-graph. 

It remains to show that it is row-finite and has no sources. That it is row-finite follows from $\Lambda$ being row-finite. To obtain a contradiction suppose that $v \in (\Lambda \setminus H)^0$ is a source. Then there is a $m \in \mathbb{N}^{\mathbb{N}}_{0}$ so that $v(\Lambda \setminus H)^m = \emptyset. $ But we know that $v\Lambda^m \neq \emptyset.$ So it must be that $s(v\Lambda^m) \subseteq H.$ Since $H$ is saturated this gives $v \in H.$ We have a contradiction. So $v$ is not a source. Thus $\Lambda \setminus H$ is a row-finite $\mathbb{N}$-graph with no sources.

So by Theorem \ref{limit} $C^{*}(\Lambda \setminus H) \cong \overline{\bigcup_{k \in \mathbb{N}}C^{*}(^k(\Lambda \setminus H))}.$
It is shown in the proof of Theorem \ref{lattice} if $H$ is saturated and hereditary in $\Lambda$ then it is saturated and hereditary in ${}^k\Lambda$. Therefore ${}^k\Lambda \setminus H ={}^k(\Lambda \setminus H) $. So by \cite[Theorem 5.2]{RaeburnSimsYeend} we know $C^{*}({}^k\Lambda \setminus H) \cong C^{*}({}^k\Lambda) / I(H)_k.$ Thus we get $C^{*}(\Lambda\setminus H) \cong \overline{\bigcup_{k \in \mathbb{N}}C^{*}(^k(\Lambda \setminus H))} \cong \overline{\bigcup_{k \in \mathbb{N}}C^{*}(^k(\Lambda)) / I(H)_k} \cong C^{*}(\Lambda) / I(H).$
\end{proof}

The following lemmas will be used in the next section. These results however are interesting to the ideal structure in their own right. They follow closely of the work of \cite{BrownFullerPittsReznikoff}.
\begin{lemma}
Let $\Lambda $ be a row-finite $\mathbb{N}$-graph with no sources. Let $J$ be an ideal in $C^{*}(\Lambda).$ Then $I(H(J))$ is the largest gauge invariant ideal contained in $J$.
\end{lemma}
Proof same as in \cite[Theorem 5.6]{Schenkel}.
\begin{lemma}
Let $\Lambda $ be a row-finite $\mathbb{N}$-graph with no sources. Let $J$ be an ideal in $C^{*}(\Lambda).$ If $\Lambda \setminus H(J)$ is aperiodic then $J$ is gauge invariant.
\end{lemma}
Proof same as in \cite[Theorem 5.7]{Schenkel}.
\section{Regular ideals of $C^*(\Lambda)$}
We now switch away from the traditional route of which these constructions take to look at a more specific question. The similarities between the $k$-graph $C^*$-algebras and $\mathbb{N}$-graph papers likely goes farther than we have examined in this paper. Though there are many routes to go at this point, we cover ground which is perhaps a little more familiar to this author. We examine the structure of regular gauge-invariant ideals of $C^{*}(\Lambda)$ It should be no surprise at this point that the proofs and techniques are largely similar to those used in the case of $k$-graph algebras. We follow the work of \cite{Schenkel} in this section.

\begin{definition}
An ideal $J$ in an algebra $A$ is called regular if $J^{\perp \perp} = J$ where $J^{\perp} = \{a \in A : ax = xa = 0$ $\forall$ $x \in J\}$.
\end{definition}
We note that if $J$ is an ideal then $J^{\perp}$ is a regular ideal.

\begin{proposition}
Let $J$ be a gauge invariant ideal in a $\mathbb{N}$-graph $C^{*}$-algebra. Then $J^{\perp}$ is gauge invariant.
\end{proposition}
Proof is the same as in \cite[Lemma 6.1]{Schenkel}.

We will need the following notation:
\begin{definition}
\begin{enumerate}
    \item[(i)] For $w \in \Lambda^{0}$, put $T(w) = \{ s(\lambda) : \lambda \in \Lambda, r(\lambda) = w \}$.
    \item[(ii)] If $ I \subseteq C^{*}(\Lambda)$ an ideal, let  $\overline{H}(I) \subseteq \Lambda^{0}$ be the set   $$\overline{H}(I) = \{ r(\lambda) : \lambda \in \Lambda \text{ and } s(\lambda) \in H(I) \}$$.
\end{enumerate} 
\end{definition}

The following gives a vertex set description for the regular, gauge-invariant ideals. It is in slightly less generality than \cite{Schenkel} as we have yet to do work in the locally-convex setting.
\begin{theorem}
Let $\Lambda$ be a row-finite $\mathbb{N}$-graph with no sources. Let $J \subseteq C^{*}(\Lambda)$ be a gauge-invariant ideal. Then: 
\begin{enumerate}
    \item[(i)] $J^{\perp} = I( \Lambda^0 \setminus \bar{H} (J) )$;
    \item[(ii)] $J^{\perp \perp} = I ( \{ w \in \Lambda ^{0} : T(w) \subseteq \bar{H}(J) \}  );$ and
    \item[(iii)] $J$ is regular if and only if $H(J) = \{ w \in \Lambda ^{0} : T(w) \subseteq \bar{H}(J) \}$.
\end{enumerate}

\end{theorem}
The proof is the same as that of \cite[Theorem 6.3]{Schenkel}.

As with $k$-graph algebras we note that regular ideals are nice to use when taking quotients. Specifically, they preserve aperiodicity. 
\begin{theorem}\label{regquotient}
 Let $\Lambda$ be a row finite $\mathbb{N}$-graph with no sources which is aperiodic. Let $J$ be a gauge-invariant regular ideal of $C^{*}(\Lambda)$ then $ \Lambda / J $ is aperiodic.
\end{theorem}

\begin{proof}
Take $v$ to be a vertex in $H(J^{\perp})$. Then there exists $\lambda \in v \Lambda$ such that $d(\lambda) \geq m \vee n$ and $\lambda(m,m+d(\lambda)-(m \vee n)) \neq \lambda (n, n+ d(\lambda) - (m \vee n))$. Since $\lambda$ is a path with $r(\lambda) = v \in H(J^{\perp})$, it follows that $\lambda(i,i) \in H(J^{\perp})$ for all $i \leq d(\lambda) \in \mathbb{N}^{k}$. Thus we have that $q(\lambda)$ satisfies that in $\Lambda/J$ for $q(v)$ that $q(\lambda) \in v(\Lambda/J)$ such that $d(q(\lambda)) \geq m \vee n$ and $$q(\lambda)(m, m+d(q(\lambda))-(m \vee n)) \neq q(\lambda)(n, n+d(q(\lambda))-(m \vee n))$$
Now take $w$ to be a vertex in $ \overline{H}(J)$. Then there is a path $\gamma$ with $r(\gamma) = w$ and $s(\gamma) \in H(J^{\perp})$. We can choose $\gamma$ so that $\gamma(i,i) \neq \gamma(j,j)$ for all $i \neq j$. Thus the path $\gamma \lambda_{s(\gamma)}$ satisfies the aperiodicity condition at $w$.
\end{proof}

The proofs of the following are the same as \cite[Lemma 6.6, Proposition 6.7, Corollary 6.8]{Schenkel} respectively, so we omit them here.

\begin{lemma}
Let $\Lambda$ be a row-finite $\mathbb{N}$-graph with no sources. Let $J$ be a regular ideal of $C^{*}(\Lambda)$. Then $I(H(J)) \subseteq J$ is a regular gauge-invariant ideal.
\end{lemma}
\begin{proposition}
If $\Lambda$ is an aperiodic, row-finite $\mathbb{N}$-graph with no sources, and $J$ is a regular ideal in $C^{*}(\Lambda)$, then $J$ is gauge-invariant.
\end{proposition}
\begin{corollary}\label{61}
 Let $\Lambda$ be an aperiodic, row-finite $\mathbb{N}$-graph with no sources. Let $J$ be a regular ideal in $C^{*}(\Lambda).$ Then $\Lambda \setminus J$ is aperiodic and $C^{*}(\Lambda) / H(J) \cong C^{*}(\Lambda \setminus H(J)).$
\end{corollary}

\section{Infinite rank graph algebras}
As was the case with $k$-graph $C^*$-algebras to Kumjian-Pask algebras, we are also interested in the question of finding an algebraic analogue to our $\mathbb{N}$-graph $C^*$-algebras. We explore in the rest of this paper the algebraic structure we call $\mathbb{N}$-graph algebras. We are particularly interested in proving as much as possible to align with what was shown in the $C^*$-algebra case. There are some interesting similarities and differences when moving into this setting.
We follow largely the work of \cite{Kumjian Pask Algebra} as much of their work carries into this algebraic analogue.
\begin{definition}
Let $\Lambda$ be an $\mathbb{N}$-graph.
Define $G(\Lambda):= \{ \lambda^{*}:\lambda \in \Lambda \}$, and call each $\lambda^{*}$ a \emph{ghost path}. If $v\in \Lambda^{0}$, then we identify $v$ and $v^{*}$. We extend the degree functor $d$ and the range and source maps $r$ and $s$ to $G(\Lambda)$ by $d(\lambda^{*})=-d(\lambda)$, $r(\lambda^{*})=s(\lambda)$ and $s(\lambda^{*})=r(\lambda)$. We extend the factorization property to the ghost paths by setting $(\mu \lambda)^{*}= \lambda^{*} \mu^{*}$. We denote by $\Lambda^{\neq 0}$ the set of paths which are not vertices and by $G(\Lambda^{\neq 0})$ the set of ghost paths that are not vertices.
\end{definition}

\begin{definition}
Let $ \Lambda$ be a row-finite $\mathbb{N}$-graph and let $R$ be a commutative ring with 1. A Kumjian-Pask $\Lambda$-family $(P, S)$ in an $R$-algebra $A$ consists of two functions $P: \Lambda ^{0} \mapsto A$ and $S: \Lambda ^{\neq 0} \cup G(\Lambda ^{\neq 0}) \mapsto A$ such that:
\begin{enumerate}
    \item[(KP1)] $\{P_{v}:v \in \Lambda^{0}\}$ is a family of mutually orthogonal idempotents;
    \item[(KP2)] for all $\lambda, \mu \in \Lambda ^{\neq 0}$ with $r(\mu) =s(\lambda)$, we have $S_{\lambda}S_{\mu} = S_{\lambda \mu} ,\  S_{\mu^{*}}S_{\lambda^{*}}= S_{(\lambda \mu)^{*}},\  P_{r(\lambda)}S_{\lambda}= S_{\lambda}=S_{\lambda}P_{s(\lambda)},  
    \text{ and } P_{s(\lambda)}S_{\lambda ^{*}}=S_{\lambda ^{*}}=S_{\lambda ^{*}}P_{r(\lambda)}$;
    \item[(KP3)]  for all $n\in \mathbb{N}^{\mathbb{N}}_{0} \setminus \{0\}$ and $\lambda , \mu \in \Lambda^{n}$, we have $S_{\lambda ^{*}}S_{\mu} = \delta_{\lambda , \mu } P_{s(\lambda)}$;
    \item[(KP4)] for all $v \in \Lambda^{0} $ and all $n \in \mathbb{N}^{\mathbb{N}}_{0} \setminus \{ 0 \}$, we have $P_{v} = \sum _{ \lambda \in v \Lambda ^{n}}S_{\lambda}S_{\lambda ^{*}}$.
\end{enumerate}

\end{definition}

\begin{theorem}
Let $\Lambda$ be a $\mathbb{N}$-graph. Then for $\lambda , \mu \in \Lambda$ and $q \in \mathbb{N^{N}}_{0}$ with $d(\lambda) , d(\mu) \leq q$ we have that $$s_{\lambda}^{*}s_{\mu} = \sum _{\lambda \alpha = \mu \beta \text{ and } d(\lambda \alpha) = q}s_{\alpha}s_{\beta}^{*} .$$
\end{theorem}
Proof is the same as in \cite[Lemma 3.1]{KumjianPask}.
\begin{theorem}
Let $ \Lambda$ be a  row-finite $\mathbb{N}$-graph with no source.
\begin{enumerate}
\item[(i)] There is a unique $R$-algebra $KP_{R}(\Lambda)$, generated by a Kumjian–Pask $\Lambda$-family $(p,s)$, such that if $(Q,T)$ is a Kumjian–Pask $\Lambda$-family in an $R$-algebra $A$, then there exists a unique $R$-algebra homomorphism $\pi_{Q,T} : KP_{R}(\Lambda) \mapsto A $ such that $\pi_{Q,T} \circ p = Q$ and $\pi_{Q,T} \circ s = T$. For every $r\in R \setminus \{ 0 \} $ and $v \in \Lambda^{0}$, we have $rp_{v} \neq 0$. 
\item[(ii)] The subsets $KP_{R}(\Lambda)_{n} := span \{ s_{\alpha}s_{\beta^{*}} : d(\alpha)-d(\beta)= n \}$ form a $\mathbb{Z^N}_0$-grading of $KP_{R}(\Lambda)$.
\end{enumerate}
\end{theorem}
The proof follows in the steps of \cite[Theorem 3.4]{Kumjian Pask Algebra}. Consider the free algebra $\mathbb{F}_R (w(X))$ on $ X := \Lambda^0 \cup \Lambda^ {\neq 0} \cup G(\Lambda^{\neq 0}.$ As in \cite{Kumjian Pask Algebra} let $I$ be the ideal generated by the ideal of $ \mathbb{F}_R (w(X))$ generated by the union of the following sets: 
\begin{enumerate}
    \item $\{vw - \delta_{v,w} v : v, w \in \Lambda^0 \}$;
    \item $\{\lambda - \mu \nu, \lambda^* - \nu^* \mu^* : \lambda, \mu \nu \in \Lambda^0 \text{ and } \lambda = \mu \nu \} \cup \{r(\lambda) \lambda - \lambda, \lambda - \lambda s(\lambda), s(\lambda^*) - \lambda^*, \lambda^* - \lambda^* r(\lambda): \lambda \in \Lambda^0\} $;
    \item $ \{\lambda^* \mu - \delta_{\lambda, \mu} s(\lambda) : \lambda , \mu \in \Lambda ^{\neq 0} \}  $;
    \item  $\{ v - \sum_{\lambda \in v\Lambda^n} \lambda \lambda^* : v \in \Lambda^0, n \in \mathbb{N}^{\mathbb{N}}_{0}\setminus \{0\}\}$.
\end{enumerate}

Now define $KP_{R}(\Lambda) := \mathbb{F}_R (w(X)) / I$. 
With this construction the proof follows exactly the same as that of \cite[Theorem 3.4]{Kumjian Pask Algebra}.

We now show that $KP_R(\Lambda)$ is the direct limit of $\{KP_R({}^k\Lambda)\}.$ It is an interesting fact but not necessary for the remainder of the proofs in this section.
\begin{lemma}\label{45}
Suppose $\Lambda$ is a row-finite $\mathbb{N}$-graph with no sources, and that there is an algebra $B$ with a Kumjiam-Pask $\Lambda$ family. Then there is a Kumjian-Pask ${}^k\Lambda$-family in $B$.
\end{lemma}
\begin{proof}
Suppose $ (P,S)$ is a Kumjian-Pask $\Lambda$-family. Define $(P,S)_k := \{P_v , S_\lambda : v \in \Lambda^0, \lambda \in {}^k\Lambda\}$.
Then $(P,S)_k$ is a Kumjian-Pask ${}^k\Lambda $-family. 
Indeed, KP1-KP3 hold by ${}^k\Lambda$ being a subcategory. As for KP4, it holds from choice of $n \in \mathbb{N}^k$ giving all of the necessary elements contained in ${}^k\Lambda.$ So it is indeed a C-K ${}^k\Lambda$ family.
\end{proof}
\begin{corollary}\label{2}
 Suppose $i \leq k$ and $\Lambda$ is a row-finite $\mathbb{N}$-graph with no sources, and that there is an algebra $B$ with a Kumjiam-Pask ${}^k\Lambda$ family. Then there is a Kumjian-Pask ${}^i\Lambda$-family in $B$.
\end{corollary}
Proof follows the same as in Lemma \ref{45}.
\begin{lemma}
Let $\Lambda$ be a row-finite $\mathbb{N}$-graph with no sources. The family of $\{KP_R({}^k\Lambda)\}$ form a directed system under canonical inclusion maps.
\end{lemma}
\begin{proof}
By Corollary \ref{2} there is a Kumjian-Pask ${}^i\Lambda$-family in $KP_R({}^k\Lambda). $ Thus there is a homomorphism $\iota_{i,k}: KP_R({}^i\Lambda) \rightarrow KP_R({}^k\Lambda)$ so that $\iota_{i,k}(s_\lambda) = s_\lambda. $ Consider the subalgebra of $KP_R({}^k\Lambda)$ generated by the Kumjian-Pask ${}^i\Lambda$-family in $KP_R({}^k\Lambda)$. This subalgebra is $\mathbb{Z}^i$ graded by the degree map. Thus $\iota_{i,k}$ is a $\mathbb{Z}^i$-graded homomorphism, as degree maps are the same for each of $s_\lambda \in {}^i\Lambda$. Thus by the graded uniqueness theorem \cite[Theorem 4.1]{Kumjian Pask Algebra}, $\iota_i,k$ is injective. The rest is now clear.
\end{proof}
\begin{theorem}
 Let $\Lambda$ be a row-finite $\mathbb{N}$-graph with no sources. Then $KP_R(\Lambda)\cong \bigcup_{k \in \mathbb{N}} KP_R({}^k\Lambda)$.
\end{theorem}
\begin{proof}
We know that $lim_{\rightarrow} KP_R({}^k\Lambda) = \bigsqcup_k KP_R({}^k\Lambda) \setminus \sim $. From our inclusion maps being injective and the equivalence relations this gives us that \\$lim_{\rightarrow} KP_R({}^k\Lambda) = \bigcup_k KP_R({}^k\Lambda).$ Thus we need to show that this is the universal algebra. Take $\{t_\lambda: \lambda \in \Lambda \}$ a Kmjian-Pask $\Lambda$-family. Then $\{t_\lambda : \lambda \in {}^k\Lambda$ is a Kumjian-Pask ${}^k\Lambda$-family. So by the universal property of $KP_R({}^k\Lambda)$ there is a homomorphism $\pi_k: KP_R(\Lambda) \rightarrow \{t_\lambda : \lambda \in {}^k\Lambda\}.$ This is true for each $k$. These mappings will compose with the inclusion map. As each $ \lambda \in \Lambda$ belongs to ${}^k\Lambda$ for some $k$, these mappings create a homomorphism from $\bigcup_k KP_R({}^k\Lambda) \rightarrow \{t_\lambda : \lambda \in \Lambda\}$. Thus it is the universal algebra and $\bigcup_k KP_R({}^k\Lambda) \cong KP_R(\Lambda).$
\end{proof}
\section{Uniqueness theorems for $KP_R(\Lambda)$}
In this section we give graded and Cuntz-Krieger uniqueness theorems. The proofs of which are the same as those in \cite{Kumjian Pask Algebra}. This is an interesting difference in switching to the algebraic setting, as the ones for the $C^*$-algebra case relied heavily on the use of the inductive limit.

We use the following lemmas inherently in our uniqueness proofs as they are necessary in the work of \cite{Kumjian Pask Algebra}.

\begin{lemma}
Every nonzero $x \in KP_R(\Lambda)$ can be written as a sum \\$\sum_{(\alpha, \beta) \in F} r_{\alpha, \beta}s_{\alpha} s_{\beta}^{*}$ where $F$ is a finite subset of $\Lambda \times \Lambda, r_{\alpha, \beta} \in R \setminus \{0\}$ for all $(\alpha, \beta) \in F$, and all the $\beta$ have the same degree. In this case we say $x$ is written in normal form. 
\end{lemma}
Proof is the same as \cite[Lemma 4.2]{Kumjian Pask Algebra}

\begin{lemma}
Suppose $x$ is a nonzero element of $KP_R(\Lambda)$ and \\$x = \sum_{(\alpha, \beta) \in F} r_{\alpha, \beta} s_{\alpha}s_{\beta}^{*}$ is in normal form. Then there exists $\gamma \in F_2 := \{ \beta : (\alpha, \beta) \in F \text{ for some } \alpha \in \Lambda\}$ such that $0 \neq x s_{\gamma} = \sum_{\alpha \in G} r_{\alpha, \gamma} s_{\alpha}$ where $G := \{ \alpha : (\alpha, \gamma) \in F \},$ Further if $\delta \in G$ then $0 \neq s_\delta^* x s_{\gamma} = r_{\delta, \gamma} p_{s(\delta)} + \sum_{\alpha \in G} r_{\alpha, \gamma} s_{\delta}^* s_{\alpha}$ and $r_{\delta, \gamma}p_{s(\delta)}$ is the $0$-graded component of $s_{\delta}^*xs_\gamma .$
\end{lemma}
Proof is the same as that of \cite[Lemma 4.3]{Kumjian Pask Algebra}.
\begin{theorem}
Let $\Lambda$ be a row-finite $\mathbb{N}$-graph with no sources, $R$ a commutative ring with $1$ and $A$ a $\mathbb{Z^N}_0 $ graded ring. If $\pi: KP_R(\Lambda) \mapsto A$ is a $\mathbb{Z^N}_0$-graded ring homomorphism such that $\pi(p_v) \neq 0 $ for all $r \in R\setminus \{0\}$ and $v \in \Lambda^0$ then $\pi$ is injective.
\end{theorem}
Proof is the same as that of \cite[Theorem 4.1]{Kumjian Pask Algebra}.
We remind the reader of the definition of a $\mathbb{N}$-graph being aperiodic.
\begin{definition}

We say $\Lambda$ is \emph{aperiodic} if for each vertex $v \in \Lambda^{0}$ and each pair $n \neq m \in \mathbb{N^{N}}_{0}$ there is a path $\lambda \in v \Lambda$ such that $d(\lambda) \geq m \vee n$ and $$\lambda(m, m+d(\lambda)-(m \vee n)) \neq \lambda(n, n+ d(\lambda)-(m \vee n))$$ 
\end{definition}

The following is a rephrasing of Lemma \ref{111}.
\begin{lemma}
Let $(\Lambda, d) $ be an aperiodic $\mathbb(N)$-graph with no sources. Suppose that $v \in \Lambda^0$ and $l \in \mathbb{N}^{\mathbb{N}}_{0}$. Then there exists $\lambda \in \Lambda$ such that $r(\lambda)=v, d(\lambda) \geq l$ and $ \alpha, \beta \in \Lambda v , d(\alpha), d(\beta) \leq l $ and $\alpha \neq \beta $ implies $(\alpha \lambda)(0,d(\lambda)) \neq (\beta \lambda)(0,d(\lambda)).$
\end{lemma}
Proof is the same as in \cite[Lemma 6.2]{HazlewoodRaeburnSimsWebster}.
\begin{proposition}
Let $\Lambda$ be an aperiodic row-finite $\mathbb{N}$-graph without sources and let $R$ be a commutative ring with 1. Let $x = \sum_{(\alpha,\beta) \in F} r_{\alpha, \beta} s_{\alpha} s_{\beta}^*$ be a nonzero element of $KP_R(\Lambda)$ in normal form. Then there exists $\sigma, \tau \in \Lambda, (\delta, \gamma) \in F$ and $w \in \Lambda^0$ such that $s_{\sigma}^* x s_{\tau} = r_{\delta, \gamma}p_w.$
\end{proposition}

\begin{theorem}
Let $\Lambda$ be an aperiodic, row-finite $\mathbb{N}$-graph with no sources, let $R$ be a commutative ring with $1$, and let $A$ be aring. If $\pi: KP_R(\Lambda) \mapsto A$ is a ring homomorphism such that $\pi(rp_v) \neq 0$ forall $r \in R \setminus \{0\}$ and $v \in \Lambda^0,$ then $\pi$ is injective.
\end{theorem}
Proofs are the same as \cite[Proposition 4.9]{Kumjian Pask Algebra} and \cite[Theorem 4.7]{Kumjian Pask Algebra} respectively.

\section{Ideal structure of $KP_R(\Lambda)$}
As was the case with the uniqueness theorems we are also able to get the structure of the graded ideals without using our inductive limit. Again this marks a difference in the analytic and algebraic settings. This work follows \cite{Kumjian Pask Algebra}.
\begin{definition}
We say an ideal $I$ is \emph{basic} If $rp_v \in I$ implies $p_v \in I$.
\end{definition}
\begin{lemma}
Let $\Lambda$ be a row-finite $\mathbb{N}$-graph with no sources. Let $I$ be an ideal of $KP_R(\Lambda)$. Then $H(I) := \{ v \in \Lambda^0 : p_v \in I\},$ is a saturated and hereditary set.
\end{lemma}
Proof is the same as that of \cite[Lemma 5.2]{Kumjian Pask Algebra}.
\begin{lemma}
Let $\Lambda$ be a row-finite $\mathbb{N}$-graph with no sources, and $H$ a saturated and hereditary subset of $\Lambda^0$. Then $\Lambda \setminus H$, the small category with objects $ \Lambda^{0} \setminus H$, and morphisms $ \{ \lambda \in \Lambda : r(\lambda) \text{ and } s(\lambda) \in \Lambda^{0} \setminus H(I) \}$, with the factorization property $d$ inherited from $(\Lambda, d)$.
    is a row-finite $\mathbb{N}$-graph with no sources, and if $(Q,T)$ is a kumjian-Pask family for $\Lambda \setminus H$ in an $R$-algebra $A$, then $$ P_v = 
    \begin{dcases}Q_v & \text{ if } v \notin H\\
    0 & \text{ otherwise}
    \end{dcases}$$
    $$ S_\lambda = 
    \begin{dcases}T_\lambda & \text{ if } s(\lambda) \notin H\\
    0 & \text{ otherwise}
    \end{dcases}$$
    $$ S_\mu ^* = 
    \begin{dcases}T_\mu ^* & \text{ if } s(\mu) \notin H\\
    0 & \text{ otherwise}
    \end{dcases}$$
form a Kumjian-Pask $\Lambda$-family $(P,S)$ in $A$.
\end{lemma}
The proof of $\Lambda \setminus H$ being a row-finite $\mathbb{N}$-graph with no sources is the same as $C^*$-algebra case, Theorem \ref{quotient}. The rest of the proof follows as in \cite[Lemma 5.3]{Kumjian Pask Algebra}.

\begin{lemma}
Let $\Lambda$ be a row-finite $\mathbb{N}$-graph with no sources. Let $H$ be a saturated Hereditary subset of $\Lambda^0.$ Then $$ I(H) = span_R \{s_\lambda s_\mu ^* : s(\lambda) = s(\mu) \in H \},$$ and $I(H)$ is a basic, graded and idempotent ideal of $KP_R(\Lambda),$ and $H(I(H)) = H.$
\end{lemma}

\begin{lemma}
Let $\Lambda$ be a row-finite $\mathbb{N}$-graph with no sources, and $R$ a commutative ring with $1$. Let $I$ be a basic graded ideal of $KP_{R}(\Lambda)$, and let $(q,t)$ and $(p,m)$ be the universal Kumjian-Pask families in $KP_{R}(\Lambda \setminus H(I))$ and $KP_{R}(\Lambda)$, respectively. Then there exists an isomorphism $\pi : KP_{R}(\Lambda \setminus H(I)) \mapsto  KP_{R}(\Lambda)/I$ such that \[ \pi(q_{v}) = p_{v} + I, \pi(t_{\lambda}) = m_{\lambda} + I, \text{ and } \pi(t_{\mu^{*}}) = m_{\mu^{*}} + I    \] for $v \in \Lambda^{0} \setminus H(I) $ and $\lambda , \mu \in \Lambda^{\neq 0} \cap  s^{-1}(\Lambda^{0}/H(I))$.
\end{lemma}

\begin{theorem}
Let $\Lambda$ be a row-finite $\mathbb{N}$-graph with no sources, and $R$ a commutative ring with $1$. Then the map $H \mapsto I(H)$ is a lattice isomorphism from the lattice of saturated, hereditary subsets of $\Lambda^0$ onto the lattice of basic, graded ideals of $KP_R(\Lambda).$
\end{theorem}

Proofs are the same as that of \cite[Lemma 5.4]{Kumjian Pask Algebra}, \cite[Proposition 5.5]{Kumjian Pask Algebra} and \cite[Theorem 5.1]{Kumjian Pask Algebra} respectively.

\begin{theorem}
Let $\Lambda$ be a row-finite $\mathbb{N}$-graph with no sources, and $R$ a commutative ring with $1$. Let $J$ be a basic ideal in $KP_{R}(\Lambda)$, then $I(H(J))$ is the largest basic, graded ideal contained in $J$.
\end{theorem}

\begin{theorem}
 Let $\Lambda$ be a row-finite $\mathbb{N}$-graph with no sources, and $R$ a commutative ring with $1$. Let $J$ be a basic ideal in $KP_{R}(\Lambda)$, if $  \Lambda \setminus H(J) $ is aperiodic then $J$ is basic and graded.
\end{theorem}

Proofs follow \cite[Theorems 3.8 and 3.9]{Schenkel} respectively.
\section{Regular ideals of $KP_R(\Lambda)$}
There are again many options for where to proceed. We opt here to examine the basic, graded regular ideals of our $\mathbb{N}$-graph algebras. We recall the definition of a regular ideal. 
\begin{definition}
An ideal $J$ in an algebra $A$ is called regular if $J^{\perp \perp} = J$ where $J^{\perp} = \{a \in A : ax = xa = 0$ $\forall$ $x \in J\}$.
\end{definition}

\begin{lemma}
Let $\Lambda$ be a row-finite $\mathbb{N}$-graph with no sources, and $R$ a commutative ring with $1$. If $J$ is a graded ideal of $KP_{R}(\Lambda)$ then $J^{\perp}$ is a graded ideal.
\end{lemma}

\begin{lemma}
Let $\Lambda$ be a row-finite $\mathbb{N}$-graph with no sources, and $R$ a commutative ring with $1$.
If $J$ is a basic graded ideal of $KP_{R}(\Lambda)$ then $J^{\perp}$ is a basic graded ideal.
\end{lemma}
The proofs follow the same as in \cite[Lemma 4.1, 4.2]{Schenkel} respectively.
\begin{definition}
\begin{enumerate}
    \item[(i)] For $w \in \Lambda^{0}$, put $  T(w) = \{ s(\lambda) : \lambda \in \Lambda, r(\lambda) = w \} $ .
    \item[(ii)] If $ I \subseteq KP_{R}(\Lambda)$ an ideal, let  $\overline{H}(I) \subseteq \Lambda^{0}$ be the set \[ \overline{H}(I) = \{ r(\lambda) : \lambda \in \Lambda \text{ and } s(\lambda) \in H(I) \} .\]
\end{enumerate} 
\end{definition}
\begin{lemma}
Let $\Lambda$ be a row-finite $\mathbb{N}$-graph with no sources, and $R$ a commutative ring with $1$.
Let $H$ be a hereditary and saturated subset of $\Lambda^{0}$. Let $J$ be the ideal generated by $H$. Then $$ H(J^{\perp}) = \{v \in \Lambda^{0} : v \Lambda H = \emptyset\} = \Lambda^{0} \setminus \bar{H}(J).$$
\end{lemma}

\begin{theorem}
Let $\Lambda$ be a row-finite $\mathbb{N}$-graph with no sources, and $R$ a commutative ring with $1$.
Let $J \subseteq KP_{R}(\Lambda)$ be a basic graded ideal. Then: 
\begin{enumerate}
    \item[(i)] $J^{\perp} = I( {\Lambda^0 \setminus \bar{H} (J)});$
    \item[(ii)] $J^{\perp \perp} = I({ \{ w \in \Lambda ^{0} : T(w) \subseteq \bar{H}(J) \}  });$ and
    \item[(iii)] $J$ is regular if and only if $H(J) = \{ w \in \Lambda ^{0} : T(w) \subseteq \bar{H}(J) \}$
\end{enumerate}
\end{theorem}
The proofs follow as in \cite[Lemma 4.3, Theorem 4.4]{Schenkel} respectively.
\begin{theorem}
 Let $\Lambda$ be a row finite $\mathbb{N}$-graph with no sources which is aperiodic. Let $J$ be a basic, graded regular ideal of $KP_R(\Lambda)$ then $ \Lambda / J $ is aperiodic.
\end{theorem}
Proof is the same as $C^*$-algebra case, Theorem \ref{regquotient}.

The last three proofs follow the same as \cite[Lemma 4.8, Prop. 4.9, Cor. 4.10]{Schenkel} respectively, so we omit them here.
\begin{lemma}
Let $\Lambda$ be a row-finite $\mathbb{N}$-graph with no sources. Let $J$ be a basic, regular ideal of $KP_R(\Lambda)$. Then $I(H(J)) \subseteq J$ is a basic, graded regular ideal.
\end{lemma}
\begin{proposition}
If $\Lambda$ is an aperiodic, row-finite $\mathbb{N}$-graph with no sources, and $J$ is a basic, regular ideal in $KP_R(\Lambda)$, then $J$ is graded.
\end{proposition}
\begin{corollary}
 Let $\Lambda$ be an aperiodic, row-finite $\mathbb{N}$-graph with no sources. Let $J$ be a basic, regular ideal in $KP_R(\Lambda).$ Then $\Lambda \setminus J$ is aperiodic and $KP_R(\Lambda) / J \cong C^{*}(\Lambda \setminus H(J)).$
\end{corollary}

\end{document}